\documentclass[a4paper]{article}
\usepackage[english]{babel}
\usepackage{amsmath}
\title{Cyclic theories}
\author{Olivia Caramello\thanks{Supported by a CARMIN IH\'ES-IHP post-doctoral position} \textrm{ }and Nicholas Wentzlaff}
\date{June 20, 2014}

\usepackage[all]{xy}
\usepackage[T1]{fontenc}
\usepackage{tikz,color,makeidx, enumerate, stmaryrd}
\usetikzlibrary{matrix,arrows}

\mathcode`\.="313A  
\mathchardef\semicolon="603B 
\mathchardef\gt="313E
\mathchardef\lt="313C

\newcommand{\cod}
 {{\rm cod}}

\newcommand{\comp}
 {\circ}

\newcommand{\Cont}
 {{\bf Cont}}

\newcommand{\Gmod}
{\mathbb{G}\mbox{-mod}(\Set)}

\newcommand{\Gemod}
{\mathbb{G}_{e}\mbox{-mod}(\Set)}

\newcommand{\N}
{\mathbb{N}}

\newcommand{\Z}
{\mathbb{Z}}

\newcommand{\Q}
{\mathbb{Q}}

\newcommand{\R}
{\mathbb{R}}

\newcommand{\T}
{\mathbb{T}}

\newcommand{\G}
{\mathbb{G}}

\renewcommand{\O}
{\mathbb{O}}

 \newcommand{\kgen}
 {k-\rm{Gen}_{\vec{z}}(\vec{x}, \vec{y})}

\newcommand{\dom}
 {{\rm dom}}

\newcommand{\empstg}
 {[\,]}

\newcommand{\epi}
 {\twoheadrightarrow}

\newcommand{\hy}
 {\mbox{-}}

\newcommand{\im}
 {{\rm im}}

\newcommand{\imp}
 {\!\Rightarrow\!}

\newcommand{\Ind}[1]
 {{\rm Ind}\hy #1}

\newcommand{\mono}
 {\rightarrowtail}
 
\newcommand{\name}[1]
 {\mbox{$\ulcorner #1 \urcorner$}}

\newcommand{\ob}
 {{\rm ob}}

\newcommand{\op}
 {^{\rm op}}

\newcommand{\Set}
 {{\bf Set }}

\newcommand{\Sh}
 {{\bf Sh}}

\newcommand{\sh}
 {{\bf sh}}

\newcommand{\Sub}
 {{\rm Sub}}

\usepackage{mathrsfs}
\usepackage{amsmath,amsthm}
\usepackage{amssymb}

\newtheorem{theorem}{Theorem}[section]
\theoremstyle{definition}
\newtheorem{definition}[theorem]{Definition}
\newtheorem{lemma}[theorem]{Lemma}

\newtheorem{remarks}[theorem]{Remarks}
\newtheorem{remark}[theorem]{Remark}
\newtheorem{notations}[theorem]{Notations}

\newtheorem{example}[theorem]{Example}

\begin{document}

\maketitle
\begin{abstract}
We describe a geometric theory classified by Connes-Consani's epicylic topos and two related theories respectively classified by the cyclic topos and by the topos $[{\mathbb
N}^{\ast}, \Set]$. 
\end{abstract}

\tableofcontents

\section{Introduction}

This paper was motivated by the recent research course ``\emph{Le site \'epicyclique}'' held by A. Connes at the \emph{Coll\`ege de France}. The epicyclic topos was introduced, amongst other things, in the course as a natural setting, refining that provided by the cyclic topos, for studying the local factors of L-functions attached to arithmetic varieties through cohomology and non-commutative geometry. The content of Connes' lectures, which were based on joint work with C. Consani, should be shortly publicly available in written form (private communication); previous papers by Connes and Consani in connection with this research programme are \cite{CC2}, \cite{CC3}, \cite{CC4} and \cite{CC5}. 

In this paper, we study the cyclic and epicyclic toposes from a logical point of view, by using the techniques developed in \cite{OCPT}. More specifically, we describe a geometric theory classified by the epicylic topos and two related theories respectively classified by the cyclic topos and by the topos $[{\mathbb N}^{\ast}, \Set]$, where ${\mathbb N}$ is the multiplicative monoid of non-zero natural numbers, also considered by Connes and Consani (cf. for instance \cite{CC5}). Realizing these toposes as classifying toposes for logical theories is a potentially rich source of applications, which we plan to explore in a forthcoming article. 

The signature over which our theory $\T_E$ classified by the epicyclic topos will be axiomatized is essentially the one suggested by Connes during his course, namely that of oriented groupoids with the addition of a non-triviality predicate. In fact, one of the main results presented in the course was the characterization of the points of the epicyclic topos in terms of projective geometry over a semi-field of characteristic 1, using this language for describing the objects of the cyclic as well as of the epicyclic category. This characterization was obtained by means of a geometric but lengthy technical detour exploiting in particular the relationships between the epicyclic topos and the topos $[{\mathbb N}^{\ast}, \Set]$. The approach undertaken in this paper is instead of logical nature, and directly leads to explicit characterizations of the points of the toposes in question. This is achieved by identifying geometric theories which are classified by these toposes, and relies on the general framework of theories of
presheaf type comprehensively studied \cite{OCPT}. 

The plan of the paper is as follows.

In section 2, we describe a theory $\T_E$ classified by the epicyclic topos over the signature of oriented groupoids with non-triviality predicate. In section 3, we describe a theory $\T_C$ classified by the cyclic topos over a related signature, obtained from the former by adding a function symbol for formalizing the existence of unitary loops on the objects of the groupoid. In section 4, we describe a theory classified by the topos $[{\mathbb N}^{\ast}, \Set]$. As it turns out, this theory provides an axiomatization for the (non-trivial) ordered groups which are isomorphic to ordered subgroups of $({\mathbb Q}, {\mathbb Q}^{+})$, already shown in \cite{CC5} to correspond to the points of the topos $[{\mathbb N}^{\ast}, \Set]$.

\section{The epicyclic theory}\label{epic}

In this section we describe a theory classified by the epicyclic topos. 

The two crucial facts on which we will build our analysis are the following: 

\begin{enumerate}
\item The epicyclic category, originally introduced by Goodwillie in \cite{goodwillie}, can be realized as a category of oriented groupoids (i.e., groupoids with a notion of positivity on arrows which is preserved by composition and satisfied by all the identity arrows, cf. Definition \ref{def:ordered_groupoids} below) and order-preserving groupoid homomorphisms between them which are injective on loops (Connes-Consani, 2014, cf. Definition \ref{def:epicyclic} below);

\item Theorem 6.29 in \cite{OCPT} states that for any theory $\mathbb T$ of presheaf type (i.e., classified by a presheaf topos) and any full subcategory $\cal A$ of its category f.p.$\mathbb{T}\mbox{-mod}(\Set)$ of finitely presentable models, there exists a unique quotient ${\mathbb T}_{\cal A}$ of ${\mathbb T}$ classified by the subtopos $[{\cal A}, \Set]$ of the classifying topos ${\cal E}_{\mathbb T}=[\textrm{f.p.}\mathbb{T}\mbox{-mod}(\Set), \Set]$ of $\mathbb T$.
\end{enumerate}

Fact 1 allows to identify the epicyclic category as a full subcategory of the theory of partially oriented groupoids with non-triviality predicate introduced below in Definition \ref{def:ordered_groupoids}. Since this theory is of presheaf type (by Lemma \ref{lem:Gofpresheaf}), Fact 2 ensures the existence of a quotient of this theory classified by the epicyclic topos. The general results of \cite{OCPT}, specifically the main characterization theorem coupled with Theorems 5.3 and 5.7, or Theorem 6.32, provide explicit but generic axiomatizations for it. Our aim in this section will be that of finding a most economical list of non-redundant axioms for this theory, by exploiting the specific combinatorics of the objects in question.

\begin{definition}
	\label{def:ordered_groupoids}
	\begin{enumerate}[(a)]
		\item Let $G_{0}$ and $G_{1}$ be two sorts (corresponding to objects and arrows respectively). The \emph{language of oriented groupoids}
	$\mathcal{L}_{\mathbb{G}}$ is obtained by adding to the usual two-sorted signature of categories $(G_{0}, G_{1};
	\dom, \cod, 1, C)$ a unary function symbol $\mbox{inv}: G_{1} \rightarrow G_{1}$ expressing the inversion of arrows, commonly written as $f^{-1}$ instead of
	$\mbox{inv}(f)$, and a unary predicate $P$ for ``positivity'', expressing the notion of orientation for arrows. 	

\item	The \emph{geometric theory of oriented groupoids} $\bar {\mathbb{G}}$ is obtained by adding to the (cartesian) theory $\mathbb{C}$ of categories the following axioms:
	\begin{enumerate}[(i)]
		\item $\top \vdash_{f^{G_{1}}} f \comp f^{-1} = 1_{\cod (f)}$
		\item $\top \vdash_{f^{G_{1}}} f^{-1} \comp f = 1_{\dom (f)}$
		\item $\top \vdash_{a^{G_{0}}} P(1_{a})$
		\item $P(f) \wedge P(f') \wedge f'' = f' \comp f \vdash_{f,f',f''} P(f'')$.
		\item $P(f) \wedge P(f^{-1}) \vdash_{f} f=1_{\dom (f)}$
		\item $\top \vdash_{f} P(f) \vee P(f^{-1})$
	\end{enumerate}
	
	We call an oriented groupoid \emph{partial} if only axioms (i)-(v) are satisfied. We denote this theory by the symbol $\G$.

\item The language of oriented groupoids \emph{with non-triviality predicate $T$}, denoted $\mathcal L_{\G_T}$, is obtained by adding to $\mathcal L_{\G}$ a predicate $T$ on arrows, whose intended interpretation is the assertion that the given arrow is a non-identical endomorphism.
	
\item	The \emph{theory} $\bar \G_T$ of \emph{oriented groupoids with (non-triviality) predicate $T$} is obtained by adding to $\bar \G$ the following axioms for $T$: 
	\begin{enumerate}[(i)]
	\setcounter{enumii}{6}
		\item $T(f) \vdash_{f} \dom (f) = \cod (f)$
		\item $T(1_{x}) \vdash_{x} \bot$
		\item $T(f) \vdash_{f} T(f^{n})$ for all $n \in \Z \setminus \{ 0 \} $ 
		\item $T(f' \comp f) \vdash_{f, f'} T(f \comp f')$
		\item $\dom (f) = \cod (f) \vdash_{f} \left( f= 1_{\dom (f)} \right) \vee T(f)$
	\end{enumerate}

	The theory obtained from $\bar \G_T$ by omitting axioms (vi) and (xi) is called the theory of \emph{partially oriented groupoids with (non-triviality) predicate $T$}, and denoted by $\G_T$.
\end{enumerate}
\end{definition}

\begin{remarks}
	\begin{enumerate}[(a)]
		\item
	In the above sequents involving compositions of arrows we have deliberately omitted the formula requiring the domains and codomains to match for the composition to be defined in order to lighten the notation. This of course is not a problem since $f \comp f'$ is defined if and only if the ternary predicate $C(f,f',f\comp f')$ is satisfied.
\item Very often we will adhere to the common intuitive notation of writing $f:a \rightarrow b$ as a shorthand for the formula $\dom(f) = a \wedge \cod(f) = b$ in three variables
	$(f^{G_1},a^{G_0},b^{G_0})$.
\end{enumerate}
\end{remarks}

From the form of the above axiomatization one may note in passing the following

\begin{lemma}
	The theory $\mathbb{G}_T$ of partially oriented groupoids (with predicate $T$) is of presheaf type.
	\label{lem:Gofpresheaf}
\end{lemma}

\begin{proof}
	Note that except for axiom (viii), all the other axioms of $\mathbb{G}_T$ are \emph{cartesian}. Using the well known fact that cartesian theories are of presheaf type, as well as Theorem 6.28 in \cite{OCPT}, stating that the property of a theory to be of presheaf type is stable under adding geometric axioms of the form $\phi \vdash_{\vec x} \bot$, the result follows immediately.
\end{proof}

Further, note that every oriented groupoid can be seen as a model of $\G_{T}$ in the obvious way. Homomorphisms between oriented groupoids, regarded as models of $\G_{T}$, however correspond precisely to those functors between them which preserve the orientation and which are \emph{injective} on endomorphisms (due to the preservation of $T$). 

\begin{example}
	\label{ex:Xn}
	Let $H$ be an (ordered) group acting on a set $X$ via $\alpha: H \times X \rightarrow X$. One can associate to it an (oriented) groupoid $H \rtimes X$, by defining its
	objects to be the elements of $X$, and its arrows $x \rightarrow x'$ to be precisely those $h \in H$ for which $\alpha(h,x) = h.x = x'$. This means
	$\mbox{Hom}(x,x') = \{ h \in H \mid h.x = x' \}$. In particular composition (if defined) is obtained by composition in $H$: $(h'.x') \comp (h.x) = (h'h).x $. Notice that
	$H \rtimes X$ adopts the orientation from the order of $H$ and the orientation is total if and only if the order of $H$ is.
\end{example}

As we will see shortly, this example is generic for our case. In fact the epicyclic category $\tilde{\Lambda}$ will be viewed as a full subcategory of
f.p.$\mathbb{G}_T\mbox{-mod}(\Set)$, that is as a full subcategory of partially oriented groupoids with predicate $T$. Notice the importance of the \emph{fullness}
condition, explaining the need for the predicate $T$. The fullness allows us to apply Theorem 6.29 \cite{OCPT} in order to show the \emph{existence} of an extension of
$\mathbb{G}_T$ classified by the epicyclic topos $[\tilde{\Lambda},\Set]$; then, by applying techniques from the same paper, we shall explicitly axiomatize it. This will be achieved by requiring the extension to be \emph{total} and the existence of a specific
factorization system for arrows.

The following alternative definition of Goodwillie's epicyclic category is due to Connes and Consani.

\begin{definition}
	\label{def:epicyclic}
	\begin{enumerate}[(a)]
	
		\item The \emph{epicyclic category} $\tilde{\Lambda}$ is the full subcategory of the category $\mathbb{G}_T\mbox{-mod}(\Set)$ consisting of the oriented
			groupoids of the form $\mathbb{Z} \rtimes X$ for \emph{transitive} $\mathbb{Z}$-actions on \emph{finite} sets $X$.
		\item The \emph{epicyclic topos} is the category $[\tilde{\Lambda}, \Set]$ of set-valued functors on the category $\tilde{\Lambda}$.
	\end{enumerate}
\end{definition}

\begin{remarks}
	\label{rem:1}
	\begin{enumerate}[(a)]
		\item Notice that every oriented groupoid $X_{n}$ in $\tilde{\Lambda}$ inherits in fact a \emph{total} orientation from $\mathbb{Z}$. In particular all the morphisms in $\tilde{\Lambda}$
			are injective on endomorphisms (in fact on hom-sets, as can be seen from remark (iii) below).
		\item In what follows we shall frequently use the following canonical representation for objects $X_{n}$ in $\tilde{\Lambda}$. Pick any of its $n$ objects and call it $x_{0}$. The transitivity of the $\mathbb{Z}$-action, ensures, 1 being the free generator of $\mathbb{Z}$, that the following
			sequence
			\[
				x_{0} \overset{1.x_{0}}{\rightarrow} x_{1} \overset{1.x_{1}}{\rightarrow} \dots \overset{1.x_{n-2}}{\rightarrow} x_{n-1} \overset{1.x_{n-1}}{\rightarrow} x_{0}
			\]
			orders all the objects in $X_{n}$, and generates every arrow in $X_{n}$. Clearly every positive integer is obtained by successive applications of
			1. From now on we shall refer to this generating $n$-loop by $\vec{\xi} = (1.x_{0},\dots,1.x_{n-1})$.
		\item In fact every arrow $x_{i} \rightarrow x_{j}$ in $X_{n}$, corresponding to $m.x_{i}$ say, will admit the following \emph{unique factorization} in terms of
			$\vec{\xi}$: Let us suppose without loss of generality that $i < j$, and let $l_{i}$ be the unique minimal positive loop based at $x_{i}$, that is $l_{i} =
			1.x_{i-1} \comp \dots \comp 1.x_{i+1} \comp 1.x_{i}$. Then there is a unique number $b \ge 0$, namely $m = (j-i) + b \cdot n$, such that 
			\[
				m.x_{i} = 1.x_{j-1} \comp \dots \comp 1.x_{i} \comp (l_{i})^{b}
			\]
			While this may look obvious, the existence of factorization systems of this form will play a crucial role in what follows.
	\end{enumerate}
\end{remarks}

Before venturing on the undertaking of constructing a geometric theory $\mathbb{T}_{E}$ in the language $\mathcal{L}_{\mathbb G_T}$ of oriented groupoids with predicate
$T$ which is classified by the \emph{epicyclic topos} (hence the subscript $E$), let us first introduce some convenient notation for simplicity, exclusively remaining in
the language $\mathcal L _{\mathbb G}$ for the rest of the section. 
\begin{notations}
	\label{not:chapter1}
Here $\vec{x}, \vec{y}$ and $\vec{z}$ denote strings of $n, m$ and $k$ variables respectively, all of sort $G_{1}$ (arrows).

 $\mathbb N$ denotes the set of natural numbers, including $0$.
\begin{enumerate}
	\item Let $\oplus_{n}$ and $\ominus_{n}$ denote addition modulo $n$ in $\{0, 1, \dots, n-1\}$. If $n$ is clear from the context we will usually omit the
		subscript, and sometimes even write $+$ or $-$ instead.
	\item Let $P(\vec{x})$ denote the extension of the positivity predicate to strings of variables, that is
		\[
			\mathbf{P}(\vec{\mathbf{x}}): \space \bigwedge_{i=1}^{n} P(x_{i})
		\]
	\item Let $L_{n}(\vec{x})$ denote the predicate expressing that $\vec{x}$ is an $n$-loop
		\[
			\mathbf{L_{n}}(\vec{\mathbf{x}}): \space \cod(x_{1}) = \dom(x_{2}) \wedge \dots \wedge \cod(x_{n}) = \dom(x_{1})
		\]
	\item Let $l(\vec{x})$ be the term defined for loops $L_{n}(\vec{x})$ and expressing the successive application of the elements of $\vec{x}$
		\[
			\mathbf{l}(\vec{\mathbf{x}}) = x_{n} \comp \dots \comp x_{2} \comp x_{1}.
		\]
		More generally, for any loop $L_{n}(\vec{x})$ and any $i\in \{1, \ldots, n\}$, we denote by $l_{i}(\vec{x})$ the $n$-loop based at $\dom x_{i}$:
		\[
			\mathbf{l_{i}}(\vec{\mathbf{x}}) = x_{i\oplus (n-1)}\comp \cdots \comp x_{i}.
		\]
	\item Let $\tilde \Phi_{n}(\vec x)$ denote the following formula:
		\[
			\mathbf {\tilde \Phi_{n}} (\vec{\mathbf{x}}): \space L_{n}(\vec x) \wedge P(\vec x) \wedge T(l(\vec x))
		\]
		We call such a $\vec x$ a positive non-trivial $n$-loop, or simply $n$-loop for short. Notice that due to axiom (viii) in Definition
		\ref{def:ordered_groupoids}, $T(l(\vec x))$ ensures the validity of $T(l_{i} (\vec x))$ for all $i$. As it will be seen shortly, the objects $X_{n} \in
		\tilde \Lambda$ are finitely presented by the formulae $\tilde \Phi_{n}$ as models of $\G_T$.
	\item 	Let $k\hy {Gen}(\vec{x}, \vec{z})$ denote the proposition, defined for $\vec{x}$ and $\vec{z}$ positive loops, expressing that $\vec{z}$ is a $k$-loop generating the loop $\vec{x}$ by successive applications
		$\dots z_{i+1} \comp z_{i} \dots$, such that two successive arrows $x_{i}$ and $x_{i+1}$ are also generated successively, i.e. if $x_{i} = z_{s_{i} \oplus
		r_{i}} \comp \dots z_{s_{i}}$ and $x_{i+1} = z_{s_{i+1} \oplus r_{i+1}} \comp \dots z_{s_{i+1}}$, then $s_{i+1} \equiv s_{i} \oplus r_{i} $ (modulo $k$) and $\dom(x_{1})=\dom(z_{1})$. This gives:
		\begin{align*}
			& \mathbf{k}\hy{\bf Gen}(\vec{\mathbf{x}}, \vec{\mathbf{z}}):  \tilde\Phi_{k}(\vec z) \wedge
			\bigvee_{\substack{p_{1},\dots,p_{n} \in \N \\ 1 = \alpha_1 \le \dots \le \alpha_n \le k }} \bigg( \bigwedge_{\substack{i \in \{
			1,\dots,n \}}} \\ & \Big( x_{i} = z_{\alpha_{i+1} -1} \comp \cdots z_{\alpha_i} \comp \cdot l_{\alpha_i}^{p_{i}}			\Big)  \bigg) 
		\end{align*}	
		where modular addition has been omitted, but of course the $\alpha_i$'s and their arithmetic lie in $\Z / k \Z$.

	\item Let $k\hy {Gen}(\vec{x}, \vec{y}, \vec{z})$ denote the proposition, defined for positive loops $\vec{x}, \vec{y}$ and $\vec{z}$ expressing that $\vec{z}$ is a $k$-loop generating both the loops $\vec{x}$ and $\vec{y}$
		successively (in the sense of the previous point) with $\dom(x_{1})=\dom(z_{1})$. We denote by $C_{n}$ the group of cyclic permutations on $n$ letters, and for $\tau \in C_{n}$ let
		$\tau(\vec x) = (x_{\tau(1)},\dots,x_{\tau(n)})$. This gives: 
		\begin{align*}
			& \mathbf{k}\hy{\bf Gen}(\vec{\mathbf{x}}, \vec{\mathbf{y}}, \vec{\mathbf{z}}):  k \hy{Gen}(\vec x , \vec z) \wedge \bigvee_{\tau \in C_{m}} k
			\hy{Gen}(\tau(\vec y), \vec z)
		\end{align*}
\end{enumerate}
\end{notations}

\begin{example}
	\label{ex:A_5}
	The simple picture to keep in mind for these generating loops, are the generators $\vec \xi$ of $X_{n}$. For example consider $A_{5}$ in $\tilde \Lambda$:

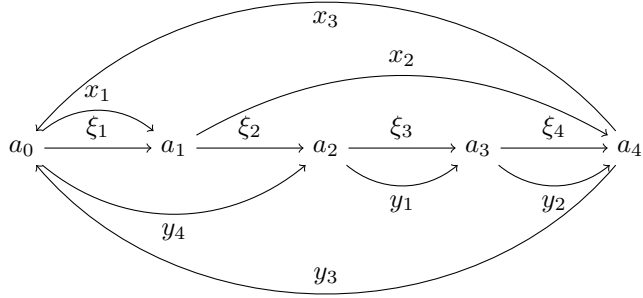
	\begin{tikzpicture}
		\node at (0,0) (a0) {$a_0$};
		\node at (2,0) (a1) {$a_1$};
		\node at (4,0) (a2) {$a_2$};
		\node at (6,0) (a3) {$a_3$};
		\node at (8,0) (a4) {$a_4$};
		\draw[->] (a0) to[out=40,in=140] node[label=above:$x_1$] {} (a1); \draw[->] (a1) to[out=30,in=150] node[label=above:$x_2$] {} (a4); \draw[->] (a4) to[out=130,in=50]
		node[label=below:$x_3$] {} (a0);
		\draw[->] (a2) to[out=-40,in=-140] node[label=below:$y_1$] {} (a3); \draw[->] (a3) to[out=-40,in=-140] node[label=below:$y_2$] {} (a4); \draw[->] (a4) to[out=-130,in=-50]
		node[label=above:$y_3$] {} (a0); \draw[->] (a0) to[out=-40,in=-140] node[label=below:$y_4$] {} (a2);
		\draw[->] (a0) to node[label=above:$\xi_1$] {} (a1); 
		\draw[->] (a1) to node[label=above:$\xi_2$] {} (a2); 
		\draw[->] (a2) to node[label=above:$\xi_3$] {} (a3); 
		\draw[->] (a3) to node[label=above:$\xi_4$] {} (a4); 
	\end{tikzpicture}

	The two loops $\vec x = (x_1, x_2, x_3)$ and $\vec y = (y_1, y_2, y_3, y_4)$ are clearly generated by $\vec \xi$, namely 
	\begin{align*}
		(x_3, x_2, x_1) &= (\xi_5, \xi_4 \comp \xi_3 \comp \xi_2, \xi_1) &\mbox{ and} \\
		(y_4, y_3, y_2, y_1) &= (\xi_2 \comp \xi_1, \xi_5, \xi_4, \xi_3).
	\end{align*}

	Notice the \emph{successive} generation in terms of $\vec \xi$. 
\end{example}

Recall that a $\T$-model $M$ is \emph{presented} by the formula $\{ \vec x. \phi\}$ if and only if the model homomorphisms $M \rightarrow N$ stand in bijective
correspondence with the interpretations $ \llbracket \vec x. \phi \rrbracket_{N}$ of the formula $\phi (\vec x)$ in the model $N$, naturally in $N$. Applied to our situation of the theory
$\G_T$ of partially oriented groupoids one readily obtains the following result:

\begin{lemma}
	The oriented groupoid $X_{n}$ as a model of $\G_T$ is \emph{finitely presented} by the formula
	\[
		\big\{ (f_{1},\dots,f_{n}) . \tilde \Phi_{n} (\vec f ) \big\}.
	\]
	\label{lem:Xn_fin_pres}
\end{lemma}

\begin{proof}
	Given a $\G_T$-model homomorphism $H: X_{n} \rightarrow G$, the string of images $H(\xi_{i})$ under $H$ of the generators	$\xi_{i} = \left( 1.x_{i-1} \right) : x_{i-1} \rightarrow x_{i}$ of $X_{n}$ satisfies the formula $\tilde \Phi_{n}$ as $H$, being a $\G_T$-model homomorphism, preserves $P$, $L_{n}$ and $T$.

	Conversely, suppose that we are given an element $\vec f$ of the interpretation of the formula $\tilde \Phi_{n}$ in a model $G$ of $\G_T$. Recalling the axioms for $\G_T$, in particular axioms (iv), (ix) and (x), we see that arbitrary compositions of $f_{i}$'s preserve $P$, that $T(l(\vec f))$ implies $T(l_{i}(\vec f))$ for all $i$, and that arbitrary non-zero powers of loops $l_{i}(\vec f)$ satisfy $T$. Now, since every positive arrow in $X_{n}$ is
	obtained by successive applications from $\vec \xi$ and every non-trivial endomorphism is obtained by non-zero powers of the primitive loops $l_{i}(\vec \xi)$ (see Remarks
	\ref{rem:1} (b) and (c)), the fact that the assignment $H(\xi_{i}) := f_{i}$ defines a $\G_T$-model homomorphism follows immediately.

	Clearly, these correspondences are inverse to each other and natural in $G$, as required.
\end{proof}

Notice that all the hitherto established results have put us in a position to exploit the techniques developed in \cite{OCPT} in order to show the existence,
and describe an axiomatization, of a theory $\T_{E}$ classified by the epicyclic topos. Indeed, from Lemma \ref{lem:Gofpresheaf} we know that the theory $\G_T$ is of presheaf
type; in particular its classifying topos can be represented as $\mathcal E _{\G_T} \simeq [\mbox{f.p.}\G_T \mbox{-mod}(\Set) , \Set]$. Also, by definition \ref{def:epicyclic}, $\tilde
\Lambda$ is a \emph{full} subcategory of $\G_T \mbox{-mod}(\Set)$, hence by Lemma \ref{lem:Xn_fin_pres} in particular of f.p.$\G_T \mbox{-mod}(\Set)$. It readily follows
from Theorem 6.29 in \cite{OCPT} that the full theory $\T_{E}$ of all geometric sequents, expressed in the language $\mathcal L _{\G_T}$, which are valid in all models
$X_{n} \in \tilde \Lambda$, is classified by the epicyclic topos.

In order to find an economical axiomatization of $\T_{E}$, all one needs to do now is to find geometric sequents in $\mathcal L_{\G_T}$ which are satisfied in every
$X_{n} \in \tilde \Lambda$ and which entail the validity of the five (schemes of) sequents in Theorem 6.32 (\cite{OCPT}). For the reader's convenience, we report the statement of the part of Theorem 6.32 which is relevant for our purposes. Note that, as it is clear from its proof, the terms referred to in the statement of the theorem are exclusively used for representing elements of the models in $\cal K$, so they could be replaced by more general $\mathbb T$-provably functional predicates if the latter are what is needed to represent such elements in terms of the generators of the model in question. In fact, in our case, the terms used are actually only `partially defined' terms (since the composition law in a groupoid is only partially defined). 

\begin{theorem}[cf. Theorem 6.32 \cite{OCPT}]\label{extending}
Let $\mathbb T$ be a theory of presheaf type over a signature $\Sigma$ and $\cal K$ a full subcategory of the category of set-based $\mathbb T$-models such that every $\mathbb T$-model in $\cal K$ is both finitely presented and finitely generated (with respect to the same generators). Then the following sequents, added to the axioms of $\mathbb T$, yield an axiomatization of the theory ${\mathbb T}_{\cal K}$ (where we denote by $\cal P$ the set of geometric formulae over $\Sigma$ which present a $\mathbb T$-model in $\cal K$): 

\begin{enumerate}[(i)]
\item The sequent 
\[
(\top \vdash_{[]} \mathbin{\mathop{\textrm{ $\bigvee$}}\limits_{\phi(\vec{x})\in {\cal P}}} (\exists \vec{x})\phi(\vec{x}));
\]

\item For any formulae $\phi(\vec{x})$ and $\psi(\vec{y})$ in $\cal P$, where $\vec{x}=(x_{1}^{A_{1}}, \ldots, x_{n}^{A_{n}})$ and $\vec{y}=(y_{1}^{B_{1}}, \ldots, y_{m}^{B_{m}})$, the sequent  
\[
(\phi(\vec{x}) \wedge \psi(\vec{y}) \vdash_{\vec{x}, \vec{y}} \mathbin{\mathop{\textrm{ $\bigvee$}}\limits_{\substack{\chi(\vec{z})\in {\cal P}, t_{1}^{A_{1}}(\vec{z}), \ldots, t_{n}^{A_{n}}(\vec{z}) \\ s_{1}^{B_{1}}(\vec{z}), \ldots, s_{m}^{B_{m}}(\vec{z})} } (\exists \vec{z})(\chi(\vec{z}) \wedge \mathbin{\mathop{\textrm{ $\bigwedge$}}\limits_{\substack{i\in \{1, \ldots, n\}, \\ {j\in \{1, \ldots, m\}}}} (x_{i}=t_{i}(\vec{z}) \wedge y_{j}=s_{j}(\vec{z})))}}),
\]
where the disjunction is taken over all the formulae $\chi(\vec{z})$ in $\cal P$ and all the sequences of terms $t_{1}^{A_{1}}(\vec{z}), \ldots, t_{n}^{A_{n}}(\vec{z})$ and $s_{1}^{B_{1}}(\vec{z}), \ldots, s_{m}^{B_{m}}(\vec{z})$ whose output sorts are respectively $A_{1}, \ldots, A_{n}, B_{1}, \ldots, B_{m}$ and such that, denoting by $\vec{\xi}$ the set of generators of the model $M_{\{\vec{z}. \chi\}}$ finitely presented by the formula $\chi(\vec{z})$, $(t_{1}^{A_{1}}(\vec{\xi}), \ldots, t_{n}^{A_{n}}(\vec{\xi}))\in [[\vec{x}. \phi]]_{M_{\{\vec{z}. \chi\}}}$ and $(s_{1}^{B_{1}}(\vec{\xi}), \ldots, s_{m}^{B_{m}}(\vec{\xi}))\in [[\vec{y}. \psi]]_{M_{\{\vec{z}. \chi\}}}$;

\item For any formulae $\phi(\vec{x})$ and $\psi(\vec{y})$ in $\cal P$, where $\vec{x}=(x_{1}^{A_{1}}, \ldots, x_{n}^{A_{n}})$ and $\vec{y}=(y_{1}^{B_{1}}, \ldots, y_{m}^{B_{m}})$, and any terms $t_{1}^{A_{1}}(\vec{y}), s_{1}^{A_{1}}(\vec{y}), \ldots, t_{n}^{A_{n}}(\vec{y}), s_{n}^{A_{n}}(\vec{y})$ whose output sorts are respectively $A_{1}, \ldots, A_{n}$, the sequent 

\begin{equation*}
\begin{split}
(\mathbin{\mathop{\textrm{ $\bigwedge$}}\limits_{i\in \{1, \ldots, n\}} (t_{i}(\vec{y})=s_{i}(\vec{y}))} \wedge \phi(t_{1}\slash x_{1}, \ldots, t_{n}\slash x_{n}) \wedge \phi(s_{1}\slash x_{1}, \ldots, s_{n}\slash x_{n}) \wedge \psi(\vec{y}) \\
\vdash_{\vec{y}} \mathbin{\mathop{\textrm{ $\bigvee$}}\limits_{\chi(\vec{z})\in {\cal P}, u_{1}^{B_{1}}(\vec{z}), \ldots, u_{m}^{B_{m}}(\vec{z})}} ((\exists \vec{z})(\chi(\vec{z}) \wedge \mathbin{\mathop{\textrm{ $\bigwedge$}}\limits_{j\in \{1, \ldots, m\}}} (y_{j}=u_{j}(\vec{z}))),
\end{split}
\end{equation*}
where the disjunction is taken over all the formulae $\chi(\vec{z})$ in $\cal P$ and all the sequences of terms $u_{1}^{B_{1}}(\vec{z}), \ldots, u_{m}^{B_{m}}(\vec{z})$ whose output sorts are respectively $B_{1}, \ldots, B_{m}$ and such that, denoting by $\vec{\xi}$ the set of generators of the model $M_{\{\vec{z}. \chi\}}$ finitely presented by the formula $\chi(\vec{z})$, $(u_{1}^{B_{1}}(\vec{\xi}), \ldots, u_{m}^{B_{m}}(\vec{\xi}))\in [[\vec{y}. \psi]]_{M_{\{\vec{z}. \chi\}}}$ and $t_{i}(u_{1}(\vec{\xi}), \ldots, u_{m}(\vec{\xi}))=s_{i}(u_{1}(\vec{\xi}), \ldots, u_{m}(\vec{\xi}))$ in $M_{\{\vec{z}. \chi\}}$ for all $i\in \{1, \ldots, n\}$;

\item For any sort $A$ over $\Sigma$, the sequent 
\[
(\top \vdash_{x_{A}} \mathbin{\mathop{\textrm{ $\bigvee$}}\limits_{\chi(\vec{z})\in {\cal P}, t^{A}(\vec{z})}} (\exists \vec{z})(\chi(\vec{z}) \wedge x=t(\vec{z}))),
\]
where the the disjunction is taken over all the formulae $\chi(\vec{z})$ in $\cal P$ and all the terms $t^{A}(\vec{z})$  whose output sort is $A$;

\item For any sort $A$ over $\Sigma$, any formulae $\phi(\vec{x})$ and $\psi(\vec{y})$ in $\cal P$, where $\vec{x}=(x_{1}^{A_{1}}, \ldots, x_{n}^{A_{n}})$ and $\vec{y}=(y_{1}^{B_{1}}, \ldots, y_{m}^{B_{m}})$, and any terms $t^{A}(\vec{x})$ and $s^{A}(\vec{y})$, the sequent 

\begin{equation*}
\begin{split}
(\phi(\vec{x}) \wedge \psi(\vec{y}) \wedge t(\vec{x})=s(\vec{y})
\vdash_{\vec{x}, \vec{y}}
\mathbin{\mathop{\textrm{ $\bigvee$}}\limits_{\substack{\chi(\vec{z})\in {\cal P}, p_{1}^{A_{1}}(\vec{z}), \ldots, p_{n}^{A_{n}}(\vec{z}) \\ q_{1}^{B_{1}}(\vec{z}), \ldots, q_{m}^{B_{m}}(\vec{z})} } (\exists \vec{z})(\chi(\vec{z})} \wedge \\ \wedge \mathbin{\mathop{\textrm{ $\bigwedge$}}\limits_{\substack{i\in \{1, \ldots, n\},\\ {j\in \{1, \ldots, m\}}}} (x_{i}=p_{i}(\vec{z}) \wedge y_{j}=q_{j}(\vec{z}))) },
\end{split}
\end{equation*}
where the disjunction is taken over all the formulae $\chi(\vec{z})$ in $\cal P$ and all the sequences of terms $p_{1}^{A_{1}}(\vec{z}), \ldots, p_{n}^{A_{n}}$ and $q_{1}^{B_{1}}(\vec{z}), \ldots, q_{m}^{B_{m}}(\vec{z})$ whose output sorts are respectively $A_{1}, \ldots, A_{n}, B_{1}, \ldots, B_{m}$ and such that, denoting by $\vec{\xi}$ the set of generators of the model $M_{\{\vec{z}. \chi\}}$ finitely presented by the formula $\chi(\vec{z})$, $(p_{1}^{A_{1}}(\vec{\xi}), \ldots, p_{n}^{A_{n}}(\vec{\xi}))\in [[\vec{x}. \phi]]_{M_{\{\vec{z}. \chi\}}}$ and $(q_{1}^{B_{1}}(\vec{\xi}), \ldots, q_{m}^{B_{m}}(\vec{\xi}))\in [[\vec{y}. \psi]]_{M_{\{\vec{z}. \chi\}}}$ and $t(p_{1}(\vec{\xi}), \ldots, p_{n}(\vec{\xi}))=s(q_{1}(\vec{\xi}), \ldots, q_{m}(\vec{\xi}))$ in $M_{\{\vec{z}. \chi\}}$.
\end{enumerate} 
\end{theorem}

It may be illuminating to note that these five schemes of sequents are actually just syntactic reformulations of the first two flatness and isomorphism conditions of Theorem 5.1 (\cite{OCPT}). Specifically, sequents (i),(ii) and (iii)
correspond to syntactic formulations of the conditions of Theorem 5.3, and sequents (iv) and (v) to syntactic formulations of those of Theorem 5.7. 

\begin{theorem} 
	The geometric theory $\mathbb{T}_{E}$ classified by the epicyclic topos $[\tilde{\Lambda},\Set]$ is obtained by adding to the theory $\bar \G_T$ of oriented groupoids
	the following axioms:
	\begin{enumerate}[(i)]
		\item $\top \vdash_{[]} (\exists a^{G_{0}})(a=a) $
		\item $P(f) \vdash_{f^{G_{1}}} (\exists g) (P(g) \wedge T(g \comp f) \big) $		
		\item For all $m,n \in \N$: \\
			$\tilde \Phi_{n} (\vec x) \wedge \tilde \Phi_{m}(\vec y) \vdash_{\vec{x}, \vec{y}}
			\bigvee_{k \le n+m} (\exists \vec{z}) ( k\hy {Gen}(\vec{x}, \vec{y}, \vec{z}))$ 
		\item For all $k\ge 2$ and distinct $i, j\in \{1, \ldots, k\}$:\\ 
			$\tilde \Phi_{k}(\vec z) \wedge (\dom( z_{i}) = \dom (z_{j})) \vdash_{\vec{z}} (\exists \vec{w}) \big(\tilde \Phi_{k-1}(\vec w) \wedge (k-1)\hy
			{Gen}(\vec{z},\vec{w})  \big)$
	\end{enumerate} 
	\label{thm:axiomatization}
	We call $\mathbb{T}_{E}$ the `epicyclic theory'.
\end{theorem}

\begin{remark}
	The important point in this axiomatization is that it entails the existence of a generator in which every domain occurs only \emph{once}. With axiom (iii) alone this would
	clearly not be the case. Reconsidering the two loops $\vec x$ and $\vec y$ of Example \ref{ex:A_5}

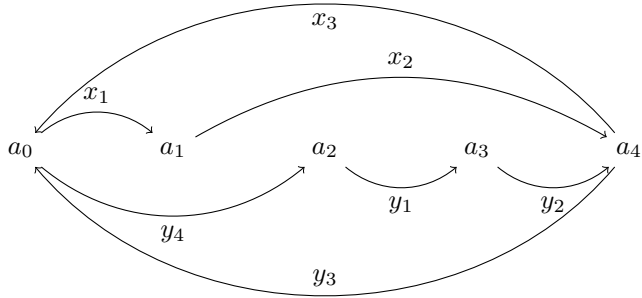
		\begin{tikzpicture}
		\node at (0,0) (a0) {$a_0$};
		\node at (2,0) (a1) {$a_1$};
		\node at (4,0) (a2) {$a_2$};
		\node at (6,0) (a3) {$a_3$};
		\node at (8,0) (a4) {$a_4$};
		\draw[->] (a0) to[out=40,in=140] node[label=above:$x_1$] {} (a1); \draw[->] (a1) to[out=30,in=150] node[label=above:$x_2$] {} (a4); \draw[->] (a4) to[out=130,in=50]
		node[label=below:$x_3$] {} (a0);
		\draw[->] (a2) to[out=-40,in=-140] node[label=below:$y_1$] {} (a3); \draw[->] (a3) to[out=-40,in=-140] node[label=below:$y_2$] {} (a4); \draw[->] (a4) to[out=-130,in=-50]
		node[label=above:$y_3$] {} (a0); \draw[->] (a0) to[out=-40,in=-140] node[label=below:$y_4$] {} (a2);
	\end{tikzpicture}
	
{\flushleft we see} that for example the loop $\vec z = (\xi_1, \xi_2, 1_{a_2}, \xi_3, \xi_4, \xi_5)$ is also a generator for $\vec x$ and $\vec y$, where just an identity has
	been added: $z_3 = 1_{a_2}$. This becomes a problem because now $\dom(z_3) = \dom(z_4)$ - a formula which does not hold for the generators $\vec \xi$. To avoid
	this, the axiom-scheme (iv) allows to reduce the length of $\vec z$ accordingly. Still, there are equivalent ways to axiomatize this property. An alternative to (iii) and (iv) can be obtained by observing that the
	generator $\vec z$ must have \emph{precisely} as many domains (arrows) as there are \emph{different} domains in the list $(\vec x, \vec y)$. For example, let us
	denote by $k'-\mbox{domeq}(\vec z)$ the assertion that $k'$ domains of arrows with different indexes in the list $\vec z = (z_1,\dots,z_k)$ are equal. The precise formula for
	this statement is 
	\[
		\bigvee_{\Lambda_{k'} \subset \{1,\dots,k\}} \Big( \bigwedge_{i \in \Lambda_{k'}}\Big( \bigvee_{j \in (\{1,\dots,k \} \setminus i)} \dom (z_i) = \dom
		(z_j)  \Big)\Big),
	\]
	where $\Lambda_{k'}$ is any subset of cardinality $k'$. Then, axioms (iii) and (iv) could be replaced by the following axiom scheme: for all $m,n \in \N$ and for all $k' \in \{1,\dots,m+n \}$
	\[ 
		\tilde \Phi_n (\vec x) \wedge \tilde \Phi_m (\vec y) \wedge k'\mbox{-domeq}(\vec x, \vec y) \vdash_{\vec x, \vec y} \exists \vec z
		((m+n-k')\mbox{-Gen}(\vec x, \vec y, \vec z)).
	\]
	Indeed, choosing $k'$ maximal in a given model does make the generator $\vec z$ minimal: every domain from $(\vec x, \vec y)$ must occur at least once in $\vec z$ by the
	requirement of a generator, and at the same time at most once by the maximal choice of $k'$.
\end{remark}

\begin{proof}
	As already remarked above, in order to prove that the theory $\T_{E}$ is classified by the epicyclic topos, we shall first verify that all its axioms are satisfied in
	every $X_{n} \in \tilde \Lambda$, and subsequently show that their validity in any $\G_T$-model entails the validity of sequents (i)-(v) of Theorem 6.32 in \cite{OCPT}.
	
	By the way, take note that ``validity'' here means validity in any model in an \emph{arbitrary} topos $\mathcal E$, and not just $\Set$. One could formulate the entire proof in the proof system for geometric logic by only manipulating sequents, but for the sake of readability and understanding, we will conduct most of the proof semantically in the standard Kripke-Joyal semantics for toposes. Due to completeness this is equivalent. So for example if we sometimes abusively speak of an ``object $a$'', with regard to a model $M$ in $\T_E\mbox{-mod}({\cal E})$, what we really mean is a generalized element of $M(G_0)$, and analogously for arrows or more general terms and formulae.

	By definition, the underlying sets of the $X_{n}$ are non-empty sets. Hence axiom (i) is always verified. Also, as already observed in Remarks \ref{rem:1}, given any two objects, successive applications of 1 will eventually yield positive arrows from one to the other whose composites are non-trivial loops, due to the transitivity of the action and the fact that $\Z$ is freely generated by 1. This shows the validity of axiom (ii) in every $X_{n}$.

	To show the validity of axiom (iii) in every $X_{n}$, suppose that two loops $\vec x$ and $\vec y$ in $X_{n}$ are given as in the premises of axiom (iii). Take $\vec z$ to be the unique loop such that $l(\vec z) = n.\dom(x_{1})$. In light of Remark \ref{rem:1} (c), all the positive arrows in $X_{n}$ are generated by $\vec{z}$; in particular $\vec x$ and $\vec y$ are generated by $\vec z$, as desired.

	The validity of axiom (iv) in $X_{n}$ can be shown as follows.  Consider the minimal $n$-loop $\vec u$ in $X_{n'}$ starting at the object $\dom (z_{1})$; then, by Remark \ref{rem:1} (c), all the domains of the arrows $z_{i}$ are equal to one of the domains $\dom(u_{j})$. Now, take $\vec v$ to be the loop obtained from $\vec u$ by successively removing any object $\dom(u_{j})$ which is not of the form $\dom(z_{i})$ for some $i$, replacing at each step the arrows $u_{j-1}$ and $u_{j}$ by their composite $u_{j}\comp u_{j-1}$. Clearly, this loop is still generating for $\vec z$ (cf. Remark \ref{rem:1} (c)) and its length is $\leq k-1$; if it is strictly smaller than $k-1$ insert identities in this loop at arbitrary places to as to increase its size to exactly $k-1$. The result will be a $(k-1)$-loop $\vec w$ satisfying the right-hand-side of axiom (iv).    
	
	Now all that remains to be shown is that the axiomatization of $\T_{E}$ ensures the validity of the five (schemes of) sequents of Theorem 6.32 \cite{OCPT}. We shall 
	undertake this task point by point, after having introduced a number of preparatory lemmas. 

\begin{lemma}\label{lem213}
		Given a loop $\tilde \Phi_{n} (\vec x)$ and a generator $k\mbox{-Gen}(\vec x, \vec z)$ with $x_{i}=u_{i}(\vec z)$ for all $i$ in a given ${\mathbb T}_{E}$-model,
		\[
			X_{k} \models \tilde \Phi_{n}(u_{1}(\vec \xi), \dots, u_{n}(\vec \xi)),
		\]
		where $\vec{\xi}$ are the generators for $X_{k}$.
		\label{lem:decomp_compatible_with_Xk}
	\end{lemma}
	
	\begin{proof}
		All that this lemma says is that when applying the decomposition terms $u_{i}(\vec z)$ to the generator $\vec \xi$ in $X_{k}$, one still obtains a positive
		non-trivial loop. 

		Since the terms $u_{i}$ are just compositions, the positivity is clear. The fact that the compositions are \emph{successive} ensures that $(u_{1}(\vec
		\xi),\dots, u_{n}(\vec \xi))$ is a loop. Indeed, the terms $u_{i}(\vec z)$ are well defined for any loop $L_{k}(\vec z)$, because if $u_{i}(\vec z) = z_{j} \comp \cdots $ ends with
		$z_{j}$, then $u_{i+1}(\vec z) = \cdots \comp z_{j+1}$ starts with $z_{j+1}$ by the form of the $u_{i}$. 

		Non-triviality of $l(u_{1}(\vec \xi),\dots,u_{n}(\vec \xi))$ easily follows from the fact that, the exponents $p_{i}$ in the definition of $k\mbox{-Gen}$ being positive or zero, the successive generation of $\vec{x}$ from $\vec{z}$ ensures that the term in $\vec{z}$ defining the composite $l(\vec{x})=u_{n}(\vec z) \comp \cdots \comp u_{1}(\vec z)$ is a non-trivial positive power of $l(\vec{z})$.
		\renewcommand{\qedsymbol}{$\Box$ (Lemma)}
	\end{proof}

	\begin{lemma}
		Given any two loops $\tilde \Phi_{n} (\vec x)$ and $\tilde \Phi_{m}(\vec y)$ in a model of $\T_{E}$, the axioms for $\T_{E}$ ensure the existence of a \emph{minimal} generator $k
		\mbox{-Gen}(\vec x, \vec y, \vec z)$ in which every domain occurs only once, i.e. $\dom (z_{i}) = \dom (z_{j})$ iff $i = j$.
		\label{lem:minimal_k-generator}
	\end{lemma}

	\begin{proof}
		The existence of a $k$-generator is ensured by (iii).  Axiom (iv) allows to inductively reduce its size to to arrive at a generator for $\vec{z}$, and hence for $\vec{x}$ and $\vec{y}$, of minimal size.
		\renewcommand{\qedsymbol}{$\Box$ (Lemma)}
	\end{proof}

	\begin{lemma}
		Given a loop $\tilde \Phi_{n} (\vec x)$ and a minimal generator $k\mbox{-Gen}(\vec x, \vec z)$ as in the previous Lemma, then the decompositions
		\[
			x_{i} = z_{j+s}\comp  \cdots z_{j} \comp l_{j}(\vec z)^{p}
		\]
		are \emph{unique}, i.e. $1 \le j,s \le k$ and $p$ are \emph{uniquely} determined by $x_{i}$. More generally, every arrow $f$ obtained by successive
		applications from $\vec z$
		\[
			f = z_{j' + s'}\comp \cdots z_{j'} \comp  l_{j'}(\vec z)^{p'}
		\]
		determines a \emph{unique} set of parameters $1 \le j', s' \le k$ and $p' \in \Z$.
		\label{lem:unique_k-decomposition}
	\end{lemma}

	\begin{proof}
		It is clear that $j$ and $j+s$ is uniquely determined by the domain and codomain of $x_{i}$, since each of them occurs precisely once in $\dom (\vec z)$ by our 
		assumption on $\vec{z}$. Further, suppose that $z_{j+s}\comp \cdots z_{j} \comp l_{j}(\vec{z})^{p} = z_{j+s} \cdots z_{j} \comp l_{j}(\vec{z})^{p'}$. Then, due to existence of inverses, $1 =
		l_{j}(\vec{z})^{(p-p')}$. But 
		\begin{align*}
			& T(l(\vec z)) & \mbox{by definition of $k$-Gen} \\
			\Rightarrow \textrm{ } & T(l_{j}(\vec z)) & \mbox{by axiom (x) in Definition \ref{def:ordered_groupoids} of }\G_{T} \\
			\Rightarrow \textrm{ } & T(l_{j} ^{q}(\vec z)) & \mbox{for all } q \neq 0 \mbox{ by axiom (ix) in Definition \ref{def:ordered_groupoids} of } \G_{T} \\
			\Rightarrow \textrm{ } & (p-p') = 0 & \mbox{by axiom (v) in Definition \ref{def:ordered_groupoids} of }\G_{T}
		\end{align*}
		Hence also $p$ is uniquely determined by $x_{i}$.
		\renewcommand{\qedsymbol}{$\Box$ (Lemma)}
	\end{proof}
	
				\begin{lemma}\label{lemcompterm}
				Suppose that $\vec z$ is a \emph{minimal generator} in the sense of Lemma \ref{lem:minimal_k-generator}, and let $t'(\vec z) = z_{i+r} \comp \cdots
				\comp z_{i} \comp l_{i}(\vec z)^{p}$ and $t''(\vec z) = z_{j+s} \comp \cdots \comp z_{j} \comp l_{j}(\vec z)^{q}$ be two terms successive in $\vec z$. Then
				\begin{enumerate}[(a)]
					\item $t'(\vec z)^{-1}$ is successive in $\vec z$, and
					\item $t''(\vec z) \comp t'(\vec z)$ is successive in $\vec z$, if defined.
				\end{enumerate}
			\end{lemma}

			\begin{proof}
				\begin{enumerate}[(a)]
					\item \begin{align*} 
							t'(\vec z) ^{-1} &= l_{i}(\vec z)^{-p} \comp z_{i}^{-1} \comp \cdots \comp z_{i+r}^{-1} \\
							&= (z_{i}^{-1} \comp \cdots \comp z_{i-1}^{-1})^{p} \comp z_{i}^{-1}\comp \cdots \comp z_{i+r}^{-1} \\
							&= z_{i}^{-1}\comp \cdots \comp z_{i+r}^{-1} \comp (z_{i+r+1}^{-1}\comp \cdots \comp z_{i+r}^{-1})^{p} \\
							&= z_{i-1} \comp \cdots \comp z_{i+r+1} \comp (z_{i+r+1}^{-1} \comp \cdots \comp z_{i+r}^{-1})^{p+1} \\
							&= z_{i-1} \comp \cdots \comp z_{i+r+1} \comp (l_{i+r+1}(\vec z))^{-p-1}
						\end{align*}
					\item Notice that the composition is defined if and only if $j \equiv i+r+1 \mbox{ (mod  $k$)}$.
						\begin{align*}
							t''(\vec z) \comp t'(\vec z) &= z_{j+s}\comp \cdots \comp z_{j} \comp l_{j}(\vec{z})^{q} \comp z_{i+r} \comp \cdots \comp z_{i} \comp
							l_{i}(\vec{z})^{p} \\
							&= z_{j+s}\comp \cdots \comp z_{i} \comp l_{i}(\vec{z})^{p+q}
						\end{align*}
				\end{enumerate}
				\renewcommand{\qedsymbol}{$\Box$ (Lemma)}
			\end{proof}

Let us now come back to the sequents of Theorem 6.32 one by one:

	\begin{enumerate}[(i)]
		\item By axiom (i) there exists an object $a$, and by axiom (ii) there exists a non-trivial loop $f$ at $a$. In particular, $\tilde \Phi_{1} (f)$, which ensures the validity of the sequent.
		\item Given two loops $\tilde \Phi_{n} (\vec x) \wedge \tilde \Phi_{m}(\vec y)$, axiom (iii) ensures the existence of a loop $\tilde
			\Phi_{k} (\vec z)$ generating the former two. 
Lemma \ref{lem213} thus allows us to conclude our thesis.	
	
	\item Given a loop $\tilde \Phi_{m}(\vec y)$, for every term $t(\vec y)$ of sort $G^{1}$ one has $\dom (t(\vec y)) = \dom (y_{j})$ for some $j$,
			and similarly for $\cod$. Indeed, $t(\vec y)$ is at most a composition of identities, $y_{j}$'s and their inverses.

			So assume that we are given terms $t_{i}(\vec y) = s_{i}(\vec y)$ $(1 \le i \le n)$ satisfying $\tilde \Phi_{n}(t_{1}(\vec y),\dots,t_{n}(\vec y))$ and
			$\tilde \Phi_{n}(s_{1}(\vec y),\dots,s_{n}(\vec y))$ respectively. Let $\vec z$ be a minimal generator $k\mbox{-Gen}(\vec t (\vec y), \vec
			y, \vec z)$ in the sense of Lemma \ref{lem:minimal_k-generator}. By Lemma \ref{lem:decomp_compatible_with_Xk} the decomposition terms of the $y_{j} = u_{j}(\vec z)$ are compatible with the $\vec \xi$ in $X_{k}$, that is
$X_{k} \models \tilde \Phi_{m}(u_{1}(\vec \xi), \dots, u_{m}(\vec \xi))$. 
			
			By Lemma \ref{lemcompterm}, the terms $t_{i}(\vec y)$ and $s_{i}(\vec y)$ are successive in $\vec{z}$ whence by Lemma \ref{lem:unique_k-decomposition} their decompositions in
			terms of $\vec z$ are \emph{unique}. This implies that $t_{i}(u_{1}(\vec z),\dots,u_{m}(\vec z)) =
			s_{i}(u_{1}(\vec z),\dots,u_{m}(\vec z))$ and hence that $t_{i}(u_{1}(\vec \xi),\dots,u_{m}(\vec \xi))=
			s_{i}(u_{1}(\vec \xi),\dots,u_{m}(\vec \xi))$, which is what we needed to ensure for the validity of sequent (iii).
		\item Suppose that we are given an arrow $f: a \rightarrow b$.  By totality of the order, either $P(f)$ or $P(f^{-1})$. Without loss of generality we assume the former. Then axiom (ii) ensures the existence of a positive arrow $g:b \to a$ such that $T(g\comp f)$. In particular, $\tilde \Phi_{2}(f,g)$. This implies the validity of sequent (iv). 
			
		\item Suppose that we are given two loops $\tilde \Phi_{n}(\vec x)$ and $\tilde \Phi_{m}(\vec y)$ and two terms $t(\vec x) = s(\vec y)$. Let $z$ be a minimal generator $k\mbox{-Gen}(\vec x, \vec y, \vec z)$ in the sense of Lemma \ref{lem:minimal_k-generator}.
			By Lemma \ref{lem:decomp_compatible_with_Xk}, the decompositions of $x_{i} = p_{i}(\vec z)$ and $y_{j} = q_{j}(\vec z)$ are compatible with
			$\vec{\xi}$ in $X_{k}$, that is
$X_{k} \models \tilde \Phi_{n}(p_{1}(\vec \xi), \dots, p_{n}(\vec \xi))$ and $X_{k} \models \tilde \Phi_{m}(q_{1}(\vec \xi), \dots, q_{m}(\vec \xi))$. 			
			 So all that is left to show is the validity of 
\[
t(p_{1}(\vec \xi),\dots,p_{n}(\vec \xi)) = s(q_{1}(\vec \xi),\dots,q_{m}(\vec \xi))
\] 
in $X_{k}$.

			Suppose first that $t(\vec x)$ and $s(\vec y)$ are of sort $G_{1}$. These terms are obtained by \emph{successive} applications from $\vec z$ by Lemma \ref{lemcompterm}, whence Lemma \ref{lem:unique_k-decomposition} ensures the equality between them, as desired.

			If $t$ and $s$ are of sort $G_{0}$ then the result follows immediately from the previous case by using the identification between objects and identical arrows on them.
			
	\end{enumerate}
		
\end{proof}

\section{The cyclic theory}

Our aim in this section is to construct a theory $\T_{C}$ classified by Connes' \emph{cyclic topos}. Recall that this topos is defined as the category $[\Lambda, \Set]$ of set-valued functors on the cyclic category $\Lambda$ originally defined in \cite{CC1}. 

Several results about the cyclic topos are already known. In an unpublished note (\cite{Moerdijk}) I. Moerdijk suggests a theory of ``abstract circles'' classified by the cyclic topos; for a complete, fully constructive, proof the reader may refer to section 8.1.1 of \cite{OCPT}. In \cite{CC3} an equivalence of categories between a category $\mathfrak{Arc}$ of \emph{archimedian sets} and that of $\Set$-valued abstract circles is constructed and used to describe the points of the cyclic topos in terms of archimedian sets.

Amongst these different characterizations of the cyclic category, we shall use Connes-Consani's description of the  epicyclic category in terms of oriented groupoids (cf. Definition \ref{def:epicyclic}) to describe $\Lambda$ as well in these terms. As we shall see shortly however, the existence of distinguished cycles in $\Lambda$ will simplify things dramatically. 

\begin{definition}
\begin{enumerate}[(a)]
\item	The language $\mathcal L_{\G_C}$ of \emph{cyclically oriented groupoids}, or \emph{oriented groupoids with cycles}, is obtained from the language of oriented groupo-ids $\mathcal L_\G$ (cf. Definition \ref{def:ordered_groupoids}) by adding a function symbol $C: G_{0} \rightarrow G_{1}$, whose intended interpretation is the assignment to an object of the generator of the cyclic group of endomorphisms on it. Again of course, the orientation on arrows will be expressed through the unary predicate $P$. 

\item	The \emph{theory of cyclically oriented groupoids} $\bar{\G}_C$ is obtained by adding to the geometric theory $\bar \G$ of Definition \ref{def:ordered_groupoids} the
	following geometric sequents:

	\begin{enumerate}[(i)]
			\setcounter{enumi}{6}
		\item $\top \vdash_{a^{G_{0}}} \dom (C_{a}) = \cod (C_{a}) = a$
		\item $\top \vdash_{a^{G_0}} P(C_a)$
		\item $C_{a} = 1_{a} \vdash_{a^{G_{0}}} \bot$
		\item $g \comp f = C_{\dom( f)} \vdash_{f,g} f \comp g = C_{\dom (g)}$
	\end{enumerate}

	As in the epicyclic case, if axiom (vi) is omitted, one obtains the theory of \emph{partially} oriented groupoids with cycles, denoted by $\G_{C}$.
	\end{enumerate}
\end{definition}

Similarly to the epicyclic case, we have the following

\begin{lemma}
	The theory $\G_C$ of partially oriented groupoids with cycles is of presheaf type.
	\label{lem:cyclic_groupoids_of_presheaftype}
\end{lemma}

\begin{proof}
	Apart from sequent (ix), the axioms of $\G_C$ are cartesian. The thesis thus follows from Theorem 6.28 \cite{OCPT}.
\end{proof}

As in the spirit of Connes-Consani's description of the epicyclic category, one obtains the following natural characterization of the cyclic category $\Lambda$ in terms of the theory $\G_C$ of partially oriented groupoids with cycles.

Let $X_n$ denote the groupoid $\Z \rtimes \{0,\dots,n-1\}$, as in the previous section (see Example \ref{ex:Xn} and Remarks \ref{rem:1}). Its orientation is again
given by the natural condition `$P(m.x_i)$ if and only if $m \ge 0$'. $X_n$ can be canonically made into a model of $\G_C$ (or even $\bar \G_C$ of course), by interpreting the elementary
cycle $C_{x_i}$ as the minimal loop $n.x_i$.  This explains the name ``elementary cycle''. Now, every endomorphism $e:x_i \rightarrow x_i$ is of the form $e=C_{x_i}^k$ for a
unique $k$. More generally, every arrow $f: x_i \rightarrow x_j$ is of the form $f = (|j-i|).x_i \comp C_{x_i}^k$ for a unique $k$.

For the rest of this section, $X_n$ will denote the just described groupoid, considered as a model of the theory $\G_C$. 

\begin{definition}
	\begin{enumerate}[(a)]
		\item The \emph{cyclic category} $\Lambda$ is the full subcategory of the category $\G_C \mbox{-mod}(\Set)$ consisting of partially oriented groupoids with
			cycles of the form $X_n$ (for $n\geq 1$).
		\item The \emph{cyclic topos} is the category $[\Lambda,\Set]$ of set-valued functors on $\Lambda$.
	\end{enumerate}
\end{definition}

\begin{remark}
	The categories $\Lambda$ and $\tilde \Lambda$ have the same objects. Clearly, every $X_n$ as a model of $\G_T$ can be seen as a model of $\G_C$ and vice versa. The difference between these categories lies in their hom-sets. Whilst homomorphisms of $\G_T$ must preserve $T$, which means sending non-trivial loops to non-trivial loops, the homomorphisms of $\G_C$ must send elementary cycles to elementary cycles. Since by the above description of the elementary cycles in $X_n$ every non-trivial loop is a power of the elementary cycle, every $\G_C$-model homomorphism $X_n \rightarrow X_{n'}$ is also a $\G_T$-model homomorphism, i.e. $\Lambda \subset \tilde \Lambda$.	
\end{remark}

Before proceeding further, we introduce two convenient notational abbreviations, in addition to the ones already defined in \ref{not:chapter1}.

\begin{notations}
	\begin{enumerate}
		\item Let $\tilde \Psi_n (\vec x)$ denote the proposition asserting that $\vec x$ is an elementary cycle
			\[ 
				\mathbf{\tilde \Psi_n}(\vec{\mathbf x}): L_n (\vec x) \wedge P(\vec x) \wedge l(\vec x) = C_{\dom (x_1)}.  
			\]
			It is important to note that due to axiom (x) of $\G_C$, $l(\vec x) = C_{\dom (x_1)}$ entails the validity of $l_i(\vec x) = C_{\dom (x_i)}$ for all
			$i$. We will see shortly that the $\G_C$-models $X_n \in \Lambda$ are finitely presented by the formulae $\tilde \Psi_n$.
		\item Given an arrow $f: a \rightarrow b$, let $\mbox{PMin}(f)$ be the proposition asserting that $f$ is the \emph{minimal positive arrow} from $a$ to $b$, i.e.
			\[
				\mathbf{\mbox{\bf PMin}(f): } \textrm{ } P(f) \wedge (\exists g: b \rightarrow a)(P(g) \wedge g \comp f = C_a)
			\]
			In other words, $\mbox{PMin}(f)$ says that $f$ is an arrow ``contained'' in the elementary cycle $C_a$.
	\end{enumerate}
\end{notations}

Just as in the epicyclic case (see Lemma \ref{lem:Xn_fin_pres}), the set $\mathcal P$ of formulas which finitely present a $\G_C$-model 
$X_n \in \Lambda$ is of a very simple form:

\begin{lemma}
	The objects $X_n$ of $\Lambda$, viewed as models in $\G_C\mbox{-mod}(\Set)$, are \emph{finitely presented} by the formulae
	\[
		\big\{ (f_1,\dots,f_n). \tilde \Psi_n \big\}.
	\]
	\label{lem:Xn_cyclic_fin_pres}
\end{lemma}

\begin{proof}
	Just as in Lemma \ref{lem:Xn_fin_pres}, one has to establish a natural bijection between the model homomorphisms $H:X_n \rightarrow G$ in $\G_C \mbox{-mod}(\Set)$ and the elements of the interpretation
	$\llbracket \vec x. \tilde \Psi_n \rrbracket_G$ of the formula $\tilde \Psi_n (\vec x)$ in the model $G$.

	Given such a homomorphism $H$, the image of the generating loop $\vec \xi$ of $X_n$ clearly belongs to the interpretation of the formula $\tilde \Psi_n (\vec x)$ in the model $G$, by the very definition of $\G_C$-model homomorphisms. 

	Conversely, suppose that $\vec f = (f_1, \dots ,f_n)$ is a loop in $G$ satisfying $\tilde \Psi_n$. We want to show that the assignment $H: \xi_i \mapsto f_i$
	defines a $\G_C$-model homomorphism. First, it is readily seen that this assignment extends to all of $X_n$ by making it compatible with composition $H(\xi_{i+1} \comp \xi_i)
	:= H(\xi_{i+1}) \comp H(\xi_i)$ and inverses, since every arrow in $X_n$ is obtained by repeatedly applying these operations to the generators $\vec \xi$. From $L_n(\vec f)$ it follows that $H$ is compatible with $\dom$ and $\cod$. Since composition preserves $P$ (axiom (iv)) and $P(\vec f)$, $H$ preserves $P$. Lastly, $H$ preserves $C$ because
	$l(\vec f) = C_{\dom (f_1)}$ and by axiom (x) this implies $l_i (\vec f) = C_{\dom (f_i)}$.

	By construction, these correspondence is inverse to each other and natural in $G$, as required.
\end{proof}

Similarly to the epicyclic case, Theorem 6.29 \cite{OCPT} ensures that the cyclic topos classifies the full theory $\T_C$ of all geometric sequents over the signature $\mathcal L_{\G_C}$ which are valid in every
$X_n \in \Lambda$, while Theorem 6.32 \cite{OCPT} provides a generic axiomatization for this theory. 

A more explicit axiomatization for $\T_C$ is provided by the following 

\begin{theorem}
	\label{thm:cyclic_axiomatization}
	The theory $\T_C$ classified by the cyclic topos can be axiomatized in the language $\mathcal L_{\G_C}$ by adding to the theory $\bar \G_C$ of oriented groupoids
	with cycles the following axioms:
\begin{enumerate}[(i)]
\item non-triviality for objects: 
	\[
		\top \vdash_{[]} (\exists a^{G_{0}})(a=a)
	\]
\item non-triviality for hom-sets :
	\[
		\top \vdash_{a, b} (\exists f: a \rightarrow b) (P(f))
	\]

\item factorization through cycles and minimal positive arrows:
	\[
		P(f)\vdash_{f:a \rightarrow b} \bigvee_{n\in {\mathbb N}}(\exists \alpha: a \rightarrow b)(\mbox{PMin}(\alpha) \wedge f=\alpha \comp (C_{a})^{n})
	\]
\end{enumerate}
	We call the theory $\mathbb{T}_{C}$ the `cyclic theory'.
\end{theorem}

\begin{proof}
	The structure of the proof will be exactly the same as for Theorem \ref{thm:axiomatization}. First we show the validity of the axioms in every $X_n$. Then we
	proceed to check the flatness and isomorphism conditions for the functors $H_{M} : \Lambda\op \rightarrow \mathcal E$ of Theorem 5.1 \cite{OCPT} by using their
	syntactic reformulations as provided by Theorem 6.32 \cite{OCPT}.

	It is immediately clear from their definition that the objects $X_n = \Z \rtimes \{0,\dots,n-1\}$ satisfy the non-triviality of objects and hom-sets conditions (cf. Definition \ref{def:epicyclic} and Remarks \ref{rem:1}). Also, the validity of the remaining axiom expressing the existence of a factorization system through cycles follows at once from Remark \ref{rem:1} (c).

	To show the entailment of the five sequents of Theorem 6.32 (\cite{OCPT}) from the above axiomatization, we will proceed one by one. Note again that here, just as in the epicyclic case, \emph{validity} means valid in models in arbitrary Grothendieck toposes, and not just in $\Set$. So in particular care must be taken in the proof when we occasionally abuse notation and speak of an object $a$, when we really mean a generalized element of the object $M(G_0)$ for a model $M$ in $\T_C \mbox{-mod}(\mathcal E)$, and similarly for arrows or more general terms or formulae. In fact, all our arguments should be interpreted in the standard Kripke-Joyal semantics for toposes.

	\begin{enumerate}[(i)]
		\item From axioms (i) of Theorem \ref{thm:cyclic_axiomatization} and (viii) of $\bar \G_C$, one deduces $ \top \vdash \exists a (P(C_a))$. This implies, in light of axiom (vii) of $\bar \G_C$, the validity of sequent
			\[
				\top \vdash \bigvee_{n \in \N} \big( \exists \vec x (\tilde \Psi_n(\vec x)) \big).
			\]
	\end{enumerate}
	
	Before proceeding further, it will be useful to note a series of lemmas:

	\begin{lemma}
		Given two objects $a$ and $b$ in a model $M$ of ${\mathbb T}_{C}$, there exists a minimal positive arrow $m_{a,b} : a \rightarrow b$ satisfying $\mbox{PMin}(m_{a,b})$.
		\label{lem:existence_posmin_arrows}
	\end{lemma}

	\begin{proof}
		From axiom (ii) we know that there exists a positive arrow $f: a \rightarrow b$. From axiom (iii) of Theorem \ref{thm:cyclic_axiomatization} it thus follows that $(\exists \alpha: a
		\rightarrow b)(\mbox{PMin}(\alpha))$.
	\end{proof}

	\begin{lemma}\label{minlemma}
		The minimal positive arrows $m_{a,b}$ of the previous lemma are either unique, or $a=b$ in which case $m_{a,b}$ can be the identity $1_a$ or the cycle
		$C_a$. Concretely, $\T_C$ entails
		\[
		(\alpha, \alpha': a \rightarrow b) \wedge \mbox{PMin}(\alpha) \wedge \mbox{PMin}(\alpha') \vdash_{\alpha,\alpha': a \rightarrow b} (\alpha = \alpha') \vee
		(a = b).
		\]
		In particular, the decomposition of axiom (iii) of Theorem \ref{thm:cyclic_axiomatization} is \emph{unique}.
		\label{lem:unique_min_pos_arrows}
	\end{lemma}

	\begin{proof}
		Suppose that we are given two objects $a$ and $b$, and positive arrows $\alpha, \alpha': a \rightarrow b$, $\beta, \beta': b \rightarrow a$ such that $\beta \comp \alpha =
		\beta' \comp \alpha' = C_a$. Then $\alpha' \comp \alpha^{-1} = \beta'^{-1} \comp \beta = \beta'^{-1} \comp C_a \comp \alpha^{-1} = (C_b)^{n}$ for some
		$n \in \Z$.
 
Now, $\alpha^{-1 }\cdot \alpha' = (C_a)^{n}$ implies $\alpha^{-1} = (C_a)^{n-1} \cdot \alpha'$. If $n > 0$ the arrow on the right hand side of this latter equality is positive and hence $n=1 $ and $\alpha = 1_a$; in particular, $a=b$. Similarly if $n<0$. If instead $n=0$ then $\alpha = \alpha'$. This proves our thesis.
	\end{proof}

	\begin{lemma}\label{lem310}
		In a model for ${\mathbb T}_{C}$, given an elementary cycle $\vec x$ satisfying $\tilde \Psi_n(\vec x)$ and an object $c$, there exists an integer $1 \le i \le n$ and positive arrows $\alpha, \beta$ such
		that $x_i = \beta \comp \alpha$ and $\dom(\beta) = c$. In particular $\tilde \Psi_{n+1}((\dots,x_{i-1},\alpha,\beta,x_{i+1},\dots))$ is verified.

		In other words, the following sequent is entailed by the axioms of $\T_C$: 
		\[ 
			\tilde \Psi_n(\vec x) \vdash_{\vec x, c} \exists \alpha,\beta \Big( P(\alpha) \wedge P(\beta) \wedge \bigvee_{1 \le i \le n} (x_i = \beta \comp
			\alpha \wedge \dom (\beta) = c) \Big) 
		\] 
	\label{lem:recurrence_for_elemcycles} 
	\end{lemma}

	\begin{proof}

		If $c=\dom(x_i)$ for some $i$, then setting $\alpha=1_{c}$ and $\beta=x_{i}$ clearly does the job. Otherwise, by applying Lemma \ref{minlemma} we see that there exists an arrow $f: \dom(x_1) \rightarrow c$ satisfying $\mbox{PMin}(f)$, whence a positive arrow $g: c \rightarrow \dom(x_1)$ such that $g \comp f = C_{\dom(x_{1})}$. Notice that, since $c\neq \dom(x_1)$, $f \neq 1_{\dom(x_1)}$ and $g \neq 1_{\dom(x_1)}$. Then
		\begin{align*}
			g \comp f &= x_n \comp \cdots \comp x_1  \\ \Rightarrow \;  g &= x_n \comp \cdots \comp x_1 \comp f^{-1}.
		\end{align*}
				Now take $1 \le i \le n$ such that
				\begin{align*}
					x_i^{-1} \comp \cdots \comp x_n^{-1} \comp g &= x_{i-1}\comp \cdots \comp x_1 \comp f^{-1} &\mbox{is still positive, but} \\
					x_{i-1}^{-1}\comp \cdots \comp x_n^{-1} \cdot g &= x_{i-2}\comp \cdots \comp x_1 \comp f^{-1} &\mbox{is not positive anymore.}
				\end{align*} 
				Note that, since $f \neq 1 \neq g$, such an $i$ must exist. Define $\alpha^{-1} = x_{i-2}\comp \cdots \comp x_1 \comp f^{-1}$ to be the
				negative arrow of the last line, and $\beta = x_{i-1}\comp \cdots \comp x_1 \comp f^{-1}$ to be the arrow in the second last line. 
				
				We see that both $\alpha$ and $\beta$ are positive, that $\dom(\beta) = \dom(f^{-1}) = \cod(f) = c$, and lastly $\beta \comp \alpha =
				x_{i-1}$.
	\end{proof}

	\begin{lemma}
				In a model for ${\mathbb T}_{C}$, given two elementary cycles $\tilde \Psi_n (\vec x)$ and $\tilde \Psi_m (\vec y)$, there exists a $\vec z$ satisfying $\tilde \Psi_k (\vec
				z)$ and containing the former two. That is, the $x_i$ and $y_j$ are obtained from $\vec z$ through successive applications:
				\begin{align*}
					& \vdots \\
					x_i &= z_{s_{i+1} -1}\comp \cdots \comp z_{s_i } \\
					x_{i+1} &= z_{s_{i+2} -1}\comp \cdots \comp z_{s_{i+1}} \\
					& \vdots
				\end{align*}
				where the same holds for the $y_j$'s, and of course the indices $s_i$ and their arithmetic are understood in $\Z / k\Z$ as usual.
								
				Moreover, without loss of generality, this $\vec z$ may be assumed not to contain identities; in particular, $\dom z_i = \dom z_j$ implies
				$i=j$, so that, in light of Lemma \ref{lem:unique_min_pos_arrows}, decompositions in terms of $\vec z$ are \emph{unique}.
				\label{lem:cyclic_generators}
	\end{lemma}

	\begin{proof}
			We shall deduce the existence of an elementary cycle $\vec z$ containing $\vec x$ and $\vec y$ from Lemma \ref{lem310} by recursion, as follows. 

			Set $\vec z^{(0)} = \vec x$.

			Given $\vec z^{(j)}$, apply Lemma \ref{lem310} to it and $\dom( y_{j+1})$ in order to obtain $\vec z^{(j+1)}$.
			 Stop after $m$ steps and set $\vec z =
			\vec z^{(m)}$.

			It is clear that all the $x_i$'s and $y_j$'s are obtained by successive applications from $\vec z$; also, by construction and Lemma \ref{lem310}, $\tilde \Psi_k (\vec z)$.

			Concerning the last part of the proposition, it is obvious that one can remove all identical arrows from $\vec z$ without affecting the satisfaction of the desired property.
			Suppose that this is done, and assume that $\dom z_i = \dom z_j =: b$ (without loss of generality $i < j$). Then $(z_{j-1} \comp \cdots \comp z_i) : b \rightarrow b$, so it can be either
			the identity or $(C_b)^n$ (for some $n >0$). Since we have removed all the identities from $\vec z$ and each of its arrows is positive, the former cannot be
			the case, so $(z_{j-1} \comp \cdots \comp z_i)=(C_b)^n$. Then $z_{i-1}\comp \cdots \comp z_j = (C_b)^{-n + 1}$, and by positivity of the left hand side $n \le 1$. So $n=1$. But then $z_{i-1} \comp \cdots
			\comp z_j = 1_{\dom(z_j)}$, which implies, since $\vec{z}$ does not contain identities, that $i = j$.
	\end{proof}
			
	In light of these lemmas, particularly \ref{lem:cyclic_generators}, the remainder of the proof is easy. Let us turn towards the remaining sequents of Theorem
	6.32:

	\begin{enumerate}[(i)]
			\setcounter{enumi}{1}
		\item With the help of Lemma \ref{lem:cyclic_generators}, the validity of sequent (ii) is almost immediate, as it ensures the existence of an elementary
			cycle $\vec z$ containing $\vec x$ and $\vec y$. In particular, $\vec{x}$ and $\vec{y}$ are obtained \emph{uniquely} through successive applications from $\vec z$:
				\begin{align*}
					& \vdots \\
					x_i &= z_{s_{i+1} -1} \cdots z_{s_i } =: t_i (\vec z)\\
					x_{i+1} &= z_{s_{i+2} -1} \cdots z_{s_{i+1}} =: t_{i+1}(\vec z) \\
					& \vdots
				\end{align*}
			What remains to be shown is that $\tilde \Psi_n (t_1 (\vec \xi),\dots,t_n (\vec \xi))$ in $X_k$, and similarly for $\vec{y}$.

		Now, since 
		\[
			\dom (t_{i+1}(\vec \xi)) = \dom (\xi_{s_{i+1}}) = \cod (\xi_{s_{i+1}-1}) = \cod (t_i(\vec \xi)) 
		\]
		$(t_i (\vec \xi))$ satisfies $L_n$. Moreover, $(t_i (\vec \xi))$ satisfies also $P$, since the $\xi_i$ are positive and positivity is stable under composition. It remains to show
		that 
		\[
			t_n (\vec \xi) \comp \cdots \comp t_1 (\vec \xi) = C_{\dom(t_1(\vec \xi))}.
		\]
		But we know that
		\[
			t_n (\vec z) \comp \dots \comp t_1 (\vec z) = \underbrace{z_{k} \comp \cdots \comp z_1}_{l(\vec z)^p} = \underbrace{x_n \comp \cdots \comp x_1}_{l(\vec x)} = C_{\dom(x_1)}
		\]
		whence $l(\vec z)^p = C_{\dom(x_1)}$. From axiom (iii) of theory $\bar{\G}_C$ and axiom (iii) of theory ${\mathbb T}_{C}$, it thus follows that $p=1$. But this is exactly what we wanted, as now
		\[
			t_n (\vec \xi) \comp \cdots \comp t_1 (\vec \xi) = l(\vec \xi) = C_{\dom(\xi_1)}.
		\]
			This shows the validity of sequent (ii).




		\item Suppose that we are given a cycle $\vec y$ satisfying $\tilde \Psi_m (\vec y)$, and terms $t_i (\vec y), s_i (\vec y)$
			($1 \le i \le n$) such that $\tilde \Psi_n (t_1 (\vec y),\dots,t_n (\vec y))$, $\tilde \Psi_n (s_1 (\vec y),\dots,s_n (\vec y))$ and $t_i
			(\vec y) = s_i (\vec y)$ for all $i$.

			Let $\vec z$ be an elementary generating cycle without identities provided by Lemma \ref{lem:cyclic_generators} for the two cycles $\vec y$ and $(t_1 (\vec y),\dots,t_n
			(\vec y)) = (s_1 (\vec y),\dots,s_n (\vec y))$. In particular, $\vec{y}$ is obtained via successive applications from $\vec z$. 		
			It is immediate to see, by using arguments similar to those in the proof of Lemma \ref{lem:decomp_compatible_with_Xk}, that $(u_1 (\vec \xi),\dots,u_m (\vec \xi))$ satisfies $\tilde \Psi_m$ in $X_k$, where $y_j = u_j (\vec z)$. 
			
			By the analogue of Lemma \ref{lemcompterm} in the cyclic setting, for each $i$ the terms $t_i(u_1(\vec z),\dots,u_m(\vec z))$ and $s_i(u_1(\vec z),\dots,u_m({\vec z}))$ are successive in $\vec z$. Lemma \ref{lem:unique_min_pos_arrows} thus yields that $t_i (u_1 (\vec \xi),\dots,u_m(\vec \xi)) = s_i (u_1(\vec \xi),\dots,u_m (\vec \xi))$ for all $i$, which is precisely what we needed to show.
			
		\item 
			Given an arrow $f: a \rightarrow b$, axiom (iii) ensures that $f = \alpha \comp (C_a)^n$ for some $\alpha$ satisfying $\mbox{PMin}(\alpha)$, whence there exists a positive arrow
			$\beta: b \rightarrow a$ such that $\beta \comp \alpha = C_a$. Then $\tilde \Psi_2 (\alpha,\beta)$ is satisfied and $f = \alpha \comp (\beta \comp
			\alpha)^n$.
		\item  
			Suppose that we are given two elementary cycles $\vec x$ and $\vec y$, and two terms $t$ and $s$ of sort $G_{1}$ satisfying $t(\vec x) = s(\vec y)$. Let $\vec z$ be an elementary generating cycle without identities provided by Lemma \ref{lem:cyclic_generators} for the two cycles $\vec x$ and $\vec y$. So $\vec x$ and $\vec y$ are obtained \emph{uniquely} by successive applications from $\vec z$, say
			$x_i = p_i (\vec z)$ and $y_j = q_j (\vec z)$. This ensures, by using arguments similar to those in the proof of Lemma \ref{lem:decomp_compatible_with_Xk}, that $(p_1 (\vec \xi),\dots,p_n (\vec \xi))$ and $(q_1 (\vec \xi),\dots,q_m
			(\vec \xi))$ satisfy $\tilde \Psi_n$ and $\tilde \Psi_m$ respectively in $X_k$.

			Lastly, to show that $t(p_1 (\vec \xi),\dots,p_n(\vec \xi)) = s(q_1 (\vec \xi),\dots,q_m (\vec \xi))$ holds in $X_k$, one argues as in point (iii) above.  			
			
			The case of terms of sort $G_0$ follows at once from that of terms of sort $G_{1}$ just considered, by using the identification between objects and identical arrows on them.
	\end{enumerate} 
\end{proof}

\section{The topos $[{\mathbb N}^{\ast}, \Set]$}

As realized by A. Connes and C. Consani, the category $\N^*$, viewed as a multiplicative monoid, and its associated topos of set-valued functors
$\widehat{\N^*} = [\N^*, \Set]$ takes a special role in the general framework $\Delta \hookrightarrow \Lambda \hookrightarrow \tilde \Lambda$. Details and proofs of the
following introductory paragraphs will be shortly available in    a paper by them (``The cyclic and epicyclic sites'' - work in progress).

The monoid $\N^*$ acts via ``barycentric subdivision'' functors $\mbox{Sd}_k : \Delta \rightarrow \Delta$ on the simplicial category $\Delta$ and its opposite
$\mbox{Sd}^*_k : \Delta\op \rightarrow \Delta\op$. In fact these actions extend to an action $\Psi_k : \Lambda \rightarrow \Lambda$ on the cyclic category $\Lambda$, by
means of which an alternative characterization of the epicyclic category as $\tilde \Lambda \simeq \Lambda \ltimes \N^*$ can be obtained.

Now, as mentioned above, one can view $\Lambda$ as a full subcategory of the category $\mathfrak{Arc}$ of archimedian sets introduced in \cite{CC3}, thereby characterizing $\tilde \Lambda$ as a full
subcategory of $\mathfrak{Arc} \ltimes \N^*$ with objects of the form $\underbar{n} = (\Z, \theta)$ where $\theta(x) = x + n + 1$. The category $\mathfrak{Arc} \ltimes
\N^*$ in turn is a full and faithful subcategory of the category $\G_T\mbox{-mod}(\Set)$ of oriented groupoids with non-triviality predicate $T$ defined in section \ref{epic}.

Indeed, consider the functor $\mbox{Mod}: \tilde \Lambda \rightarrow \N^*$ which sends objects $X_n$ to $\star \in \N^*$ and arrows $H: X_n \rightarrow X_m$ to the integer $k \ge 1$ defined by the condition that the elementary cycle $l(\vec \xi^{(n)})$ of $X_n$ gets mapped to the $k$-th power of the elementary cycle in $X_m$: $l(\vec \xi^{(m)})^k = H(l(\vec \xi^{(n)}))$. Note that $k \ne 0$ because $H$ preserves $T$, and $k \nless 0$ because $H$ preserves $P$. This functor $\mbox{Mod}$ is part of a geometric morphism, and hence one obtains a
natural ``section'' $\imath: \widehat{\N^*} \rightarrow \tilde \Lambda$ which allows to lift descriptions of points of $\N^*$ to those of $\tilde \Lambda$. This strategy has been
pursued by Connes and Consani and resulted in interesting connections to Tits' ideas of characteristic one (cf. their above-mentioned forthcoming paper ``The cyclic and epicyclic sites'', \cite{CC6} and \cite{Tits}).

In fact, the category of points of $\widehat{\N^*}$ is shown in \cite{CC5} to correspond to that of non-trivial ordered subgroups $(H,H^+)$ of the ordered group $(\Q,
\Q^+)$ and non-trivial order-preserving group homomorphisms between them (note that these are precisely the homomorphisms between such groups viewed as models of the \emph{injectivization} of the
theory of (totally) ordered groups, as defined below). These points were in turn equivalently characterized as the algebraic extensions of the tropical integers $\mathbb F = \Z_{\mbox{max}} \subset H \subset
\overline{\mathbb F} = \Q_{\mbox{max}}$.

Again, our aim is to directly describe not only the category of points, but also a theory whose classifying topos is $\widehat{\N^*}$. 

\begin{definition}
	The geometric theory of \emph{partially ordered groups} $\mathbb{O}$ is obtained by adding to the (algebraic) theory of groups (where the constant $1$ denotes the neutral element of the group) a predicate $P$ for positivity and
	the following axioms:
	\begin{enumerate}[(i)]
		\item $\top \vdash P(1)$
		\item $P(a) \wedge P(b) \vdash_{a,b} P(a \cdot b)$
		\item $P(a) \vdash_{a,c} P(c^{-1} \cdot a \cdot c)$
		\item $P(a) \wedge P(a^{-1}) \vdash_a a=1$
	\end{enumerate}

	The theory \emph{totally ordered groups}, denoted by $\overline{\O}$, is obtained by adding the axiom
	\begin{enumerate}[(i)]
			\setcounter{enumi}{4}
		\item $\top \vdash_a P(a) \vee P(a^{-1})$
	\end{enumerate}
	
	One can axiomatize the \emph{injectivization} of $\overline{\O}$ by enriching the language with a binary relation symbol $\ne$, expressing non-equality, and adding the
	following axioms:
	\begin{enumerate}[(i)]
			\setcounter{enumi}{5}
		\item $x \ne x \vdash_x \bot$
		\item $\top \vdash_{x,y} (x \ne y) \vee (x = y)$
	\end{enumerate}
	We will denote this theory by $\overline{\O_{\ne}}$. Its \emph{partial} counterpart is obtained by omitting axioms (v) and (vii) and will be denoted by
	$\O_{\ne}$.
\end{definition}

Once more, we have a series of lemmas enabling us to apply Theorem 6.29 \cite{OCPT}:

\begin{lemma}
	The theory $\O_{\ne}$ of partially ordered groups is of presheaf-type.
	\label{lem:O_of_presheaftype}
\end{lemma}

\begin{proof}
	Axioms (i)-(iv) are cartesian, and theories of presheaf-type are stable under the addition of geometric axioms of the form $\phi \vdash_{\vec{x}} \bot$.
\end{proof}

\begin{lemma}
	The ordered group $\Z$ is finitely presented, as a model of the theory $\O_{\ne}$, by the formula 
	\[
		\big\{ x . P(x) \wedge x \ne 1 \big\}.
	\]
	\label{lem:Z_fin_pres}
\end{lemma}

\begin{proof}
	$\Z$ being the free group of one generator, it is clear that any homomorphism $H:\Z \rightarrow G$ is uniquely determined by its value at 1. Since 1 is positive, its
	image must be positive too, and since $H$ is injective, it must also be non-zero.
\end{proof}

\begin{lemma}
	The category $\N^*$ is a \emph{full} subcategory of the category $\O_{\ne}\mbox{-mod}(\Set)$ of partially ordered groups with injective homomorphisms, and in particular also of its full subcategory f.p.$\O_{\ne}\mbox{-mod}(\Set)$ of finitely presentable models.
	\label{lem:N*_full_subcat}
\end{lemma}

\begin{proof}
Group homomorphisms $\Z \rightarrow \Z$ correspond precisely to integers $n$, namely $x \mapsto nx$. The requirement of preserving the order $P$ gives $n \ge 0$, and the injectivity condition of preserving $\ne$ excludes the
	case $n=0$. Hence $\O_{\ne}$-model homomorphisms of $\Z$ correspond set-wise to $\N^*$, and since $m(nx) = (mn)x$ this correspondence is functorial.

	The second part is also immediate in light of Lemma \ref{lem:Z_fin_pres}.
\end{proof}

The following theorem describes an axiomatization of the theory $\T_N$ classified by the subtopos $\widehat{\N^*}$ of the classifying topos for $\O_{\ne}$, whose existence is ensured by the two previous lemmas and Theorem 6.29 (\cite{OCPT}).

\begin{theorem}
	The geometric theory $\T_N$ classified by the topos $\widehat{\N^*}$ is obtained by adding to the injective theory of ordered groups $\overline{\O_{\ne}}$ the
	following axioms:
	\begin{enumerate}[(i)]
		\item $\top \vdash \exists x (x \ne 1)$ 		
		\item $P(x) \wedge P(y) \vdash_{x,y} \bigvee_{n,m \in \N_{0}} \exists z (P(z) \wedge (x=z^n) \wedge (y = z^m))$
	\end{enumerate}
	\label{thm:axiomatization_N*}
\end{theorem}

\begin{proof}
	Remember that the category $\N^*$ has been identified with the full subcategory of $\O_{\ne}\mbox{-mod}(\Set)$ on the model $\Z$ in the proof of lemma \ref{lem:N*_full_subcat}. It is clear that the natural order is total, so $\Z$ is a model of $\overline{\O_{\ne}}$. It also clearly satisfies the above axioms.

	To show that the axioms are also sufficient for the full (geometric) theory Th$(\Z)$ in the language of $\O_{\ne}$, we proceed once more checking the validity of sequents (i)-(v) of Theorem 6.32 \cite{OCPT}. 
	
	Note that, as observed in Remarks 5.4 (b) and 5.8 (a) in \cite{OCPT}, the validity of sequents (iii) and (v) will follow automatically from that of (ii) due to the fact that all the model homomorphisms are \emph{monic}. Nonetheless, we shall prove them explicitly below, for the sake of clarity and completeness. To this end, we note the following 

	\begin{lemma}
		The theory $\T_N$ entails the validity of the following sequent:
		\[
			(x \ne 1) \wedge P(x) \vdash_x \bigwedge_{k \in \N} x^k \ne 1.
		\]
		In other words, $\T_{N}$-groups (that is, the groups in $\T_N\mbox{-mod}(\Set)$) are torsion-free.
		\label{lem:no_cycles}
	\end{lemma}

	\begin{proof}
		If $P(x)$ then $P(x^k)$ for all $k \in \N$ and if $x \ne 1$ none of the powers $x^k$ can be the inverse to $x$. In fact this already holds for  $\overline{\O}$, which just corresponds to the well-known fact that totally ordered groups are torsion-free.
	\end{proof}

	\begin{enumerate}[(i)]
		\item Axiom (i) ensures the existence of an $(x \ne 1)$; if $P(x)$ we are done, otherwise $P(x^{-1})$ and we are done too.
		\item Given two positive $x \ne 1 \ne y$, axiom (ii) ensures the existence of a positive $z$ such that $x = z^n$ and $y = z^m$. Clearly $z$ cannot be the
			identity, and $n$ and $m$ cannot be 0. Therefore it holds in $\Z$ that $m\cdot 1 \gneq 0 \lneq n\cdot 1$, which shows validity of sequent (ii).
		\item Given a positive $y\neq 1$ and two terms $t(y)$ and $s(y)$ which are equal and strictly positive too, the latter must be positive powers of $y$ of the same exponent by Lemma \ref{lem:no_cycles}, which proves the validity of sequent (iii).
		\item  Given any $x$, if $x=1$ then by applying axiom (i) we obtain the existence of a non-zero $y$, without loss of generality positive by the totality of the order, whence $x = 1 = 1(y)$. If $x > 0$ then we are done immediately, and if $x<0$ then $x^{-1}>0$ and $x =(x^{-1})^{-1}$. This proves the validity of sequent (iv).
		\item Lastly, suppose given two positive $x,y\neq 1$ and terms $t(x) = s(y)$. By an easy induction over terms it is readily seen that this means $x^{p} = y^{q}$ with $p,q \in \Z$. Now, axiom (ii) ensures the existence of a positive $z\neq 1$ for which $x = z^n$ and $y = z^m$; therefore $z^{p+n} = z^{q+m}$ and hence $p+n = q +m$ because of non-torsion (cf. Lemma \ref{lem:no_cycles}). So $(p+n)\cdot 1 = (q+m) \cdot 1$ in $\Z$, which is all we needed.
	\end{enumerate}
\end{proof}

Applying this result to what is already known about the points of $\widehat{\N^*}$, we see that the $\T_N$-groups are precisely the ordered groups which are isomorphic to non-trivial subgroups of $\Q$. Indeed, in \cite{CC5} the following result was obtained:

\begin{theorem}
The category of points of $\widehat{\N^*}$ is equivalent to the category of ordered groups which are isomorphic to non-trivial subgroups of the additive group $(\Q,\Q_+)$ and non-trivial (injective) homomorphisms between them.
\end{theorem}

In fact, the correspondence between the $\T_N$-groups and the ordered groups which are isomorphic to non-trivial subgroups of the additive group $(\Q,\Q_+)$ can also be proved directly, as follows. 

\newpage
\begin{theorem}
	An ordered group $G$ is a $\T_N$-group if and only if it is isomorphic to a non-trivial ordered subgroup of $\Q$. 
\end{theorem}

\begin{proof}
From the fact that the theory $\T_N$ is classified by the topos $\widehat{\N^*}$ we know that every $\T_N$-group $G$ is a filtered colimit $colim(D)$ of a diagram $D:{\cal I}\to {\mathbb N}^{\ast}\hookrightarrow \T_N\mbox{-mod}(\Set)$, where $\cal I$ is a
filtered category, since every flat set-valued functor is a filtered colimit of representables. Note that the factorization through $\N^*$ implies that these diagrams simply
consist of certain arrows $n: \Z \rightarrow \Z$ in $\T_N \mbox{-mod}(\Set)$, where $n$ is the homomorphism $1 \mapsto n$ ($n$ being a non-zero natural number). 

So to show that $G$ embeds as an ordered subgroup of $\Q$ in $\T_N \mbox{-mod}(\Set)$, it suffices to construct a cocone $(\lambda_i : D(i) \to \Q)_{i \in \cal
I}$, due to the universality of $G = colim (D)$ and since every $\T_N$-model homomorphism is monic. Recalling that $D(i) = \Z$, it is important to note that $\T_N$-model homomorphisms
$\Z \to \Q$ correspond precisely to multiplication by a strictly positive fraction $\frac{p}{q}:\Z \to \Q$, $1 \mapsto \frac{p}{q}$.

Now, to construct such a cocone, fix any object $i_{0} \in \cal I$, which must exist by definition of filtered category. We define the arrow $\lambda_{i_0} : D(i) = \Z
\to \Q$ to be the canonical embedding of $\Z$ in $\Q$. Given any $j \in \cal I$, to define a homomorphism $\lambda_j: D(j)=\Z \to \Q$ satisfying the required cocone commutation
relations we argue as follows. Using the joint embedding property of $\cal I$, there exists an object $k \in \cal I$ and arrows $i_0 \overset{f}{\rightarrow} k \overset{g}{\leftarrow} j$. We set
$\lambda_k : D(k) = \Z \to \Q$ equal to the unique arrow $\lambda_k$ such that $\lambda_{i_0} = \lambda_k \comp D(g)$. This is indeed possible because $D(g): \Z \to \Z$ is given by multiplication by a non-zero integer
$n$, say, whence $\lambda_k : \Z \to \Q$ can be taken (and this is the unique possible choice) to be multiplication by the element $\frac{1}{n}$. We then set $\lambda_j = \lambda_k \comp D(f)$; concretely, this function is given by multiplication by the element $\frac{m}{n}$, where $D(f):\Z \to \Z$, $1 \mapsto m$.

This construction is independent of the chosen $f$ and $g$, because any two parallel $g,g' : j \to k$ are weakly coequalized by some $h: k \to k'$ in $\cal I$, which
implies $D(g) = D(g')$ in $\T_N \mbox{-mod}(\Set)$ since all morphisms are monic. Finally, it is an easy exercise, using the concrete descriptions of arrows in terms of fractions of natural numbers, to show that this is independent of the chosen $k$, and that the obtained $(\lambda_i : D(i) \to \Q)_{i \in \cal I}$ is indeed a cocone.

Conversely, let us show that any subgroup $(H,H^+) \le (\Q, \Q^+)$ satisfies axiom (ii). Given two elements $x,y \in H^+$, write them as reduced fractions of
positive integers $x = \frac{x'}{a}$ and $y = \frac{y'}{b}$, i.e. $\gcd(x',a) = 1 = \gcd(y',b)$. Using B\'ezout's identity this yields $ma + nx' = 1$, for some integers $m$ and $n$, and similarly for
$y$; hence also $\frac{1}{a}, \frac{1}{b} \in H^+$. 

We are done if $z:=\frac{\gcd(a,b)}{ab} \in H^+$. Indeed, setting $g := \gcd(a,b)$, we have that if $a = rg$ and $b = sg$ for positive integers $r$ and $s$ then $rz = \frac{1}{b}$ and $sz = \frac{1}{a}$, whence $x=(x's)z$ and $y=(y't)z$. 

Now, using again B\'ezout's identity, we obtain the existence of integers $p$ and $q$ such that $g := \gcd(a,b) = pa + qb$. Then $\frac{q}{a} + \frac{p}{b} = \frac{g}{ab} = z$ is in $H^+$, which is what we needed to show.
\end{proof}

\begin{remark}
Note the similarities between the notion of $\T_N$-group and that of \emph{archimedean group}. A totally ordered group $G$ is called archimedean if for any positive $x,y \in G$ there exists a natural number $n \in \N$ such that $x \le ny$ (cf. \cite{FS}). Just as the notion of ${\mathbb T}_{N}$-group, the concept of archimedean group is formalizable within geometric logic. It follows immediately from axiom (ii) that every $\T_N$-group satisfies the archimedean property. In fact, O. H\"older has shown that an ordered group $G$ is archimedean if and only if it is isomorphic to an ordered subgroup of $(\R,+)$ (cf. \cite{FS}).  
\end{remark}

\vspace{1cm}
\textbf{Acknolwedgements:} We warmly thank Alain Connes for useful discussions on the subject matter of this paper.


\emph{Olivia Caramello}

{\small \textsc{Institut des Hautes \'Etudes Scientifiques, 35 route de Chartres, 91440, Bures-sur-Yvette, France}\\
\emph{E-mail address:} \texttt{olivia@ihes.fr}}

\vspace{0.5cm}

\emph{Nicholas Wentzlaff}

{\small \textsc{Centre de Math\'ematiques Laurent Schwartz, \'Ecole polytechnique, 91 128 Palaiseau cedex, France}\\
\emph{E-mail address:} \texttt{n\_w9@hotmail.com}}

\end{document}